\documentclass[a4paper,11pt]{amsproc}
\usepackage[english]{babel}
\usepackage[leqno]{amsmath}
\usepackage{amssymb,amsthm}
\usepackage{amscd}
\usepackage{enumerate}
\usepackage[cmtip,all]{xy}

\newtheorem{thm}{Theorem}[section]

\newtheorem{prop}[thm]{Proposition}

\newtheorem{claim}[thm]{Claim}

\newtheorem{fact}[thm]{Fact}
\newtheorem{conj}[thm]{Conjecture}
\newtheorem{prob}[]{Problem}

\theoremstyle{definition}
\newtheorem{defi}[thm]{Definition}
\theoremstyle{remark}
\newtheorem{remark}[thm]{\bf{Remark}}
\newtheorem{nota}[thm]{\bf}

\numberwithin{equation}{section}

\newtheorem{example}[thm]{Example}

\newcommand{\la}{\longrightarrow}
\newcommand{\ha}{\hookrightarrow}

\newcommand{\ov}{\overline}

\newcommand{\Div}{\operatorname{Div}}
\newcommand{\dv}{\operatorname{div}}

\newcommand{\Spec}{\operatorname{Spec}}

\newcommand{\Pic}{\operatorname{Pic}}
\newcommand{\Tw}{\operatorname{Tw}}
\newcommand{\Jac}{\operatorname{Jac}}

\newcommand{\Prin}{\operatorname{Prin}}

\newcommand{\Picphi}{\operatorname{Pic}_{\phi}}

\newcommand{\Picphir}{\operatorname{Pic}_{\phi^*}}

\newcommand{\mdeg}{\underline{\operatorname{deg}\thinspace}}

\def\e{\epsilon}

\def\L{\mathcal L}
\def\O{\mathcal O}

\def\X{\mathcal X}
\def\Y{\mathcal Y}
\newcommand{\Z}{\mathbb{Z}}

\def\ZZ{\mathcal Z}
\newcommand{\G}{\Gamma }

\def\md{\underline{d}}

\newcommand{\Mgt}{{M_g^{\rm trop}}}

\newcommand{\Mgb}{\ov{M_g}}

\newcommand{\lp}{\operatorname{loop}}
\newcommand{\hG}{\widehat{\Gamma}}

\newcommand{\rs}{r^{\#}}

\begin{document}
\title{Algebraic and combinatorial Brill-Noether theory.}
\author{Lucia Caporaso}
 \address{Dipartimento di Matematica,
 Universit\`a Roma Tre,
 Largo S. Leonardo Murialdo 1,
 00146 Roma (Italy)}
 \email{caporaso@mat.uniroma3.it}

\date{}

\begin{abstract} The interplay between algebro-geometric and combinatorial Brill-Noether theory is studied. The Brill-Noether locus $W^r_d(\G)$ of a genus-$g$ (non-metric) graph $\G$ is shown to be non-empty if the Brill-Noether number $\rho^r_g(G)$ is non-negative, as a consequence of the analogous fact for smooth projective curves.  Similarly, the existence of a graph $\G$ for which $W^r_d(\G)$ is empty implies the emptiness of  $W^r_d(C)$ for a general curve $C$ of genus $g$.
The main tool is a refinement of  Baker's Specialization Lemma.
\end{abstract}

 \maketitle

\tableofcontents
\section{Introduction}
In this paper we investigate the interplay  between the divisor theory on algebraic curves and the divisor theory on finite graphs. Recent progress in combinatorics shows that   the analogies between the two fields are quite strong.  For divisors on graphs 
there are notions of principal divisors, linear equivalence, degree and rank, canonical class, 
  Brill-Noether loci;  the corresponding basic theory  has the same shape as   for algebraic curves, and some fundamental facts, such as the 
the Riemann-Roch formula and the Clifford inequality, hold.

We focus on some classical   theorems in Brill-Noether theory,
and their analogs for graphs. 
Precise statements    for what follows
 can be found at the beginning of Section~\ref{BNsec}.
The two basic algebro-geometric theorems  are the Existence Theorem, \cite{kempf}, \cite{KL1},  
stating that if the Brill-Noether number $\rho^r_d(g)$ is non-negative, then the variety $W^r_d(C)$ is non-empty
for every smooth projective curve $C$ of genus $g$;
and the Brill-Noether Theorem, \cite{GH},
according to which if
$\rho^r_d(g)$ is negative, then $W^r_d(C)$ is empty
for a general smooth projective curve $C$ of genus $g$. 
 
 It is thus quite natural to ask whether the same results hold for graphs.
 As far as we know, the first place where these issues have been explicitly raised is M. Baker's paper
  \cite{bakersp}. One of the main goals of that paper was to prove a remarkable result, called the 
 ``Specialization Lemma", which is somewhat technical to be explained in this introduction
  (see Section~\ref{specsec}),
 but  which can be applied to connect the Brill-Noether theory of curves  to the Brill-Noether theory of graphs, as explained   in \cite{bakersp}.
  As it turns out, the applications of the Specialization Lemma  work better for metric graphs, 
 or tropical curves, rather than for
 ordinary graphs. For instance, one of the most striking  is the fact that   the  Brill-Noether Theorem for curves follows from the existence of  one tropical curve of genus $g$ for which $W^r_d$ is empty
 (such a curve is constructed  in \cite{CDPR}).

We are in this paper interested in the Brill-Noether theory of graphs (with no metric).
 Our first step is thus to strengthen the Specialization Lemma so as to make it applicable for us.
 This consists in extending it 
 from strongly semistable, regular curves over   discrete valuation rings (as assumed in \cite{bakersp}), 
 to all one-parameter families of curves
 with regular total space and nodal singular fibers;
 and from divisors
  defined on the total space (as   in \cite{bakersp}), to   families of divisors  (i.e. sections of the Picard scheme) not necessarily gluing to a globally defined divisor.
We do that in Section~\ref{specsec} and treat a more refined version for
graphs with loops in Section~\ref{loopsec}.
 
 Then we use our refined version of the Specialization Lemma, together with the   Existence Theorem for curves, to prove   the Existence Theorem for graphs:
 see Theorem~\ref{BN+}.

 We conclude the paper with a discussion on the Brill-Noether Theorem for graphs,
 noticing, in Proposition~\ref{BNBN}, that the classical Brill-Noether Theorem for curves follows from the
 existence of a (non-metric) graph for which $W^r_d$ is empty (again, this is known thanks to
 \cite{CDPR}). 
 We also include some speculations about which graphs   are Brill-Noether general,
 i.e. have an empty $W^r_d$ whenever $\rho^r_d(g)<0$.
 There are several examples of graphs that are not Brill-Noether general, and it would be interesting to have  a  classification of them, even only for $3$-regular ones. This last problem also relates to the recently very active area of research relating moduli spaces of algebraic curves, moduli spaces  of tropical curves, and Berkovich spaces; see \cite{BPR} for example.
 There is a direct correspondence between the moduli spaces of tropical curves and of algebraic curves,
 based only on the underlying (non-metric) graphs;
 see   \cite[Th, 4.2.1]{Ctrop}. It would be interesting to understand how the
 Brill-Noether theory fits in with this correspondence.
 
The paper is organized as follows. Sections~\ref{agsec} and  \ref{divsec} recall some useful definitions and results from Algebraic Geometry (Section~\ref{agsec}) and Graph Theory (Section~\ref{divsec})
and contain no original results. In Section~\ref{specsec} we prove a first refinement of the Specialization Lemma,  Proposition~\ref{spe}. In Section~\ref{loopsec} we prove a second refinement using a more precise definition of rank, taking loops into account; see Proposition~\ref{speloop}. In Section~\ref{BNsec} we prove the Existence Theorem for graphs, and further discuss the interplay between algebraic and combinatorial Brill-Noether theory.

\
  
It is my pleasure to express my gratitude to Sam Payne, for some comments and questions
out of  which this paper grew. I also wish to thank Dan Abramovich and Matt Baker for some
useful remarks, and the referees for their accurate reports.

\section{Algebro-geometric preliminaries.}
\label{agsec}
\begin{nota}
{\it Algebraic curves.}
Unless we specify otherwise, we work over an algebraically closed field $k$;
the word ``point" stands for  closed point;
the word ``curve" stands for connected, reduced, projective one-dimensional scheme defined over $k$.

A nodal curve is a curve having at most nodes as singularities.

Let $X_0$ be a nodal curve.
We denote by $\G$ the dual graph of $X_0$, 
so that its vertex set, $V(\G)$,   is identified with the set of    irreducible components of $X_0$, and its edge set,  $E(\G)$,   is identified with the set of nodes of  $X_0$, with an edge joining two (possibly equal) vertices  if the  corresponding (possibly equal) components  intersect at the corresponding node. Note that $\G$ is an ordinary graph
(no orientation,    metric, or weight function).
We denote by
\begin{equation}
\label{irred}
X_0=\bigcup_{v\in V(\G)}C_v
\end{equation}
the decomposition of $X_0$ into irreducible components.

For a Cartier divisor $D$ (or a line bundle $L$) on $X_0$, the multidegree $\mdeg  D$ is
$$
\mdeg D:=\{\deg_{C_v}D\}_{v\in V(\G)}\in \Z^{V(\G)}
$$
where $\deg_{C_v}D$ is the degree of the restriction of $D$ to $C$. We denote $r(X_0,D):=h^0(X_0,D)-1$.
\end{nota}
\begin{nota}{\it Picard scheme.}
  \label{neron}
 Details about what follows may be found in \cite{BLR}.
  Let $\phi:\X\to B$ be a family of generically smooth curves, i.e. a projective morphism whose fibers are curves, such that $B$ contains a dense open subset, $B^*\subset B$,
  over which the fibers of $\phi$ are smooth. For $b\in B$ we denote $X_b:=\phi^{-1}(b)$.
  We assume $B$ smooth and irreducible for simplicity.
 We denote  by 
 $\phi^*:\X^*\to B^*$ the restriction of $\phi$ over $B^*$.
We have the associated (relative)  Picard scheme
$$
 \pi:\Picphi \la B.
$$
The notation $\Pic_{\X/B}$ is often  used for what we here denote by $\Picphi$; our notation, almost the same  as the one used in \cite{gac}, is more convenient for our purposes.
Denote by $\pi^0:\Picphi^0 \to B$    the (relative) Jacobian. 
So,  the fiber of $\pi$ over $b\in B$ is $\Pic X_b$ and the fiber of $\pi^0$ is 
 $\Pic^0 X_b$, the generalized Jacobian  of $X_b$, denoted often by $\Jac X_b$. Recall that for $b\in B^*$ we have
 $$
 \Pic^0 X_b=\{L\in \Pic X_b: \deg L=0\}.
 $$

Now, a ``pathology"
of the Picard scheme is that 
the morphism  $\Picphi^d\to B$ is not separated  if $\phi$ admits reducible fibers
 (see  below).
It is thus desirable to have  a separated  model  for $\Picphir^d\to B^*$ over $B$.
  \end{nota}
\begin{nota}{\it N\'eron model.}
By fundamental results of A. N\'eron (we refer to \cite{BLR} for details)
 there exists a universal solution to the above problem,   the N\'eron model, provided one restricts
to the case $\dim B=1$, which we shall henceforth assume.
For our purposes, it suffices to treat the case $d=0$. The  N\'eron model 
of $ \Picphir^0\to B^*$ is a smooth,  separated group scheme of finite type over $B$, 
here denoted by
$ 
N_{\phi}^0 \to  B,
$ 
whose restriction over $B^*$ is $ \Picphir^0\to B^*$. More exactly, $N_{\phi}^0$ is the  largest separated quotient  of $\Picphi^0 \to B$.
We are going to describe   it explicitly in the special case of interest for us.

We shall assume that $B\smallsetminus B^*$ is a unique point 
$b_0$,
so that $\phi$ has only one singular fiber, $X_0:=\phi^{-1}(b_0)$.
We shall refer to $b_0$ and $X_0$ as the special point and the special fiber.
We shall assume that $X_0$ is a nodal curve. 
 
  Now we introduce the $B$-scheme $E\to B$  defined as the schematic closure in $\Picphi^0 $ of the unit
section $B^*\to \Picphi^0 $ mapping $b\in B^*$ to $\O_{X_b}\in 
\Pic( X_b)$. The fiber  of $E$ over the special point $b_0$
is a remarkable subgroup of $\Pic^0(X_0)$, called the subgroup of $\phi$-{\it twisters} and denoted by
$\Tw_{\phi}(X_0)$,
described as follows
\begin{equation}
\label{twisters}
\Tw_{\phi}(X_0):=\{\O_{\X}(\sum_{v\in V(\G)} n_vC_v)_{|X_0} ,\ \  n_v\in \Z\}_{/\cong}\subset \Pic^0(X_0).
\end{equation}
Equivalently, the $\phi$-twisters are those  line bundles on $X_0$ which occur as specializations of the trivial line bundle $\O_{\X^*}$.

Now,  $N_{\phi}^0$ is the quotient $\Picphi^0 /E$ .	
Let us point out that the N\'eron model 
 is   compatible with  finite  \'etale   base changes, but   not    with  non-\'etale   ones
(details in   \cite[Chapter 9]{BLR} or \cite[Section 3]{cner}).
  \end{nota}
\begin{nota}
{\it Component group of the N\'eron model.}
\label{nerongp}
Consider $N_{\phi}^0 \to B$.  
Its   fiber over $b_0$  depends only on $X_0$ and on the singularities of $\X$.
We need the following   standard  terminology.
\begin{defi}
\label{regsmooth}
Let   $\phi:\X\to B$ be a flat projective morphism  
  satisfying the following properties.
 $B$ is a smooth irreducible one-dimensional quasiprojective scheme;
 $b_0\in B$ is a (closed) point.
$\X$ is a nonsingular surface.
$X_0:=\phi^{-1}(b_0)$ is a  projective curve.
 
Then we say that  $\phi$ is a {\it regular one-parameter smoothing} of $X_0$.
 \end{defi}

Assume that    $\phi:\X \to B$ is     a  regular  one-parameter   smoothing of $X_0$.
Since it turns out that  the special fiber of the N\'eron model does not depend on $\phi$,  we shall denote it by $N_{X_0}$. We have that $N_{X_0}$ is non canonically isomorphic to the disjoint union of finitely many copies of the generalized jacobian of $X_0$:
\begin{equation}
\label{Ngroup}
N_{X_0}\cong \bigsqcup_{i\in \Delta_{X_0}} (\Pic^0 X_0)_{i}
\end{equation}
where  $\Delta_{X_0}$ is a finite group, often called the {\it group of components of the N\'eron model}.
The group $\Delta_{X_0}$ has been extensively studied;
in the present situation it depends only on the intersection
product of divisors of the surface $\X$, whose definition we now recall.
Using    the notation (\ref{irred}), we have  
\begin{displaymath}
(C_v\cdot C_w): =\left\{ \begin{array}{ll}
|C_v\cap C_w|& \text {if } v\neq w,\\
\\
-|C_v\cap \overline{X_0\smallsetminus C_v}|, & \text {if } v=w.\\
\end{array}\right .
\end{displaymath}
Observe that this product depends only on $X_0$, not  on $\phi$. Let us connect with definition
(\ref{twisters}); for every $v\in V(\G)$ we have   
$$
\mdeg \  \O_{\X}(C_{v})_{|X_0}=\{(C_v\cdot C_{w})\}_{w\in V(\G)}\in \Z^{V(\G)}.
$$
\begin{remark}
\label{deg0} For every $v\in V(\G)$ and every $b\in B$ we have 
$$
\deg \  \O_{\X}(C_{v})_{|X_0}=\deg \O_{\X}(C_{v})_{|X_b}=\deg (\O_{\X})_{|X_b}=0.
$$ 
\end{remark}
Then $\Delta_{X_0}$ is the quotient of degree-$0$ multidegrees by the multidegrees of all twisters,
i.e.
\begin{equation}
\label{Delta}
\Delta_{X_0}=\frac{\md \in \Z^{V(\G)}: |\md|=0}{<\mdeg\  \O_{\X}(C_{v}),\  \forall v\in V(\G)>}
=\frac{\md \in \Z^{V(\G)}: |\md|=0}{ \mdeg (\Tw_{\phi}(X_0))}
\end{equation}
where $|\md|=\sum _{v\in V(\Gamma)}d_v$ for $\md =\{d_v\}_{v\in V(\Gamma)}$ and
$\mdeg:\Pic(X_0)\to \Z^{V(\G)}$ is the multidegree homomorphism. 

Observe that $\Tw_{\phi}(X_0)$ depends on $\phi$, whereas $\mdeg  \Tw_{\phi}(X_0)$,\  
$\Delta_{X_0}$, and hence  $N_{X_0}$, do  not.

\end{nota}

\section{Divisor theory on graphs}
\label{divsec}
\begin{nota}{\it Divisors and intersection product.}
Let $\G$ be a finite connected graph, with vertex set $V(\G)$ and edge set $E(\G)$.
The genus of $\Gamma$ is its first Betti number.
The following definitions originate from \cite{BN} and \cite{bakersp}, but we do allow loops;
 see Remark~\ref{noloop}.

The group of divisors of $\G$, denoted by $\Div (\G)$, is the free abelian group generated by its  vertices:
  $$
\Div (\G):=\{\sum_{v\in  V(\G)}  n_vv, \    n_v\in \Z \}\cong \Z^{V(\G)}.
  $$

The degree of a divisor $D=\sum  n_vv$ is defined as $\deg D:=\sum n_v$; we denote by $\Div ^d(\G)$ the set of divisors of degree $d$.
If $n_v\geq 0$ for all $v$ we
  say that $D$ is {\it effective}, and write $D\geq 0$.

There is an intersection product on $\Div(\G)$ given by linearly extending the following definition
\begin{displaymath}
(v\cdot w)=\left\{ \begin{array}{ll}
\text{number of edges joining }v  \text{ and } w&\text{ if } v\neq w\\
- \deg (v) +2 \lp(v)&\text{ if } v= w\\
\end{array}\right.
\end{displaymath}
where  $\deg(v)$ is the degree, or valency,   of  $v$, and $\lp(v)$ is the number of loops based at $v$.

\end{nota}
\begin{remark}
\label{intrk}
If $\G$ is the dual graph of the nodal curve $X_0$, we have
$$
(v\cdot w)=(C_v\cdot C_w),
$$
with  the right-hand side as  defined earlier.
\end{remark}
\begin{nota}{\it Principal divisors.}
To define principal divisor in analogy with the case of algebraic curves, one  considers the set of functions on $\G$, that is, the set  
$ 
k(\G):=\{f:V(\G) \la \Z \}.
$  
Then   the ``order" of $f\in k(\G)$ at $v\in V(\G)$ is  the following integer
$$
{\rm{ord}}_v(f):= \sum_{w\in V(\G) }(v\cdot w) f(w) \in \Z.
$$
Now, to any $f\in k(\G)$ we associate the divisor $\dv (f)\in \Div(\G)$ defined as follows:
$$
\dv (f):= \sum_{v\in V(\G)}{\rm{ord}}_v(f) v .
$$
We denote by $\Prin (\G)$ the set of all divisors of the form $\dv (f)$.
\end{nota}
\begin{defi}
For $v\in V(\G)$ let $f_{v}:V(\G)\to \Z$ be the function 
such that $f_{v}(v)=1$ and $f_{v}(w)=0$ for all $w\in V(\G)\smallsetminus v$.
We set
\begin{equation}
\label{twist}
T_{v}:=\dv (f_{v})= \sum_{w\in V(\G)}(w\cdot v) w.
\end{equation}
\end{defi}
\begin{remark}
\label{twisterrk}
Let us relate this to the situation presented in Section~\ref{nerongp}: if $\phi$ is a regular one-parameter smoothing of 
$X_0$ then by Remark~\ref{intrk} we have
\begin{equation}
\label{twistdeg}
T_{v }=\mdeg_{X_0} \O_{\X}(C_{v}).
\end{equation}

The set $\{T_v,\  \forall v\in V(\G)\}$ clearly generates $\Prin (\G)$.
Therefore we can identify $\Prin (\G)$ as the  group of all multidegrees of $\phi$-twisters,
i.e. 
$$
\Prin (\G)=\{\mdeg T,\  \forall T\in \Tw_{\phi}(X_0)\}.
$$

By Remark~\ref{deg0}, or by an easy combinatorial argument,  it follows  that principal divisors on $\G$ have degree $0$.
\end{remark}

The {\it Jacobian group} of $\G$ is 
$
\Jac (\G):=\Div ^0(\G)/\Prin (\G).
$

From the previous discussion we obtain  the following  fact,  well known in algebraic geometry
 (with a different terminology on the combinatorial side).
 \begin{fact}
Let  $\phi:\X \to B$ be a regular  one-parameter smoothing of a nodal curve $X_0$, and let $\G$ be the dual graph of $X_0$. Then there is a canonical isomorphism
$$
\Delta_{X_0}\cong \Jac(\G)
$$ between the component group of the N\'eron model of $\Pic_{\phi^*}^0\to B^*$ and the Jacobian group of $\G$.
\end{fact}
We wish to emphasize the simple but important fact that the regularity assumption on $\X$ cannot be removed.
\begin{nota}{\it Combinatorial rank.}
We now go back to the purely graph-theoretic setting.
We say that $D, D'\in \Div (\G)$ are equivalent, and write $D\sim D'$, if $D-D'\in \Prin (\G)$.
It is clear that if $D\sim D'$ then $\deg D =\deg D'$.
Set
$$
|D|:=\{E\in \Div (\G): E\geq 0,\  E\sim D\}
$$
and
\begin{displaymath}
r_{\Gamma}(D)=\left\{ \begin{array}{ll}
-1 & \text {if } |D|=\emptyset \\
\\
max\{k\geq 0: \  \forall  E\in \Div^k_+(\G)  \quad    |D-E| \neq \emptyset \} & \text {otherwise,}\\
\end{array}\right.
\end{displaymath}
where $\Div^k_+(\G)$ denotes the set of effective divisors of degree $k$.

\begin{remark}
\label{basic}
If $D\sim D'$ then $|D|=|D'|$ and $r_{\Gamma}(D)=r_{\Gamma}(D')$.

Also,   $r_{\Gamma}(D)\leq \max\{-1, \deg D\}$.
\end{remark}
\begin{remark}
\label{noloop}
It is clear that the previous definitions do not depend on the loops of $\G$. In fact,
throughout  \cite{BN} and \cite{bakersp} the authors assume that the graphs are free from loops.
A good reason for doing that is that the definition of rank given above
is somewhat ``rough", for example, it
does not satisfy the Riemann-Roch formula if $\G$  contains loops. 
Anyways, even if $\G$ has some loops, with the above  rough definition
of rank the results of the present paper continue to hold, but are less tight.  
Therefore, in Section~\ref{loopsec} we will   give a more precise definition for the rank of a divisor on a graph admitting loops,
and show that our results generalize with that definition.
\end{remark}
 \end{nota}

\section{Baker Specialization Lemma refined}
\label{specsec}
Let $\phi:\X\to B$ be a family of curves.
The associated Picard scheme $\Picphi$ may be viewed as a functor from the category of schemes over $B$ to the category  of sets; see \cite[Chapter 8]{BLR}. In particular,
  $\Picphi (B)$ denotes the set of  regular sections of $\pi:\Picphi\to B$:
  $$
  \Picphi (B)=\{\sigma:B\to \Picphi:\quad  \pi\circ\sigma=id_B \}.
  $$
 \begin{remark}
There is a natural map
\begin{equation}
\label{brauer}
\Pic (\X) \la \Picphi (B);\quad \quad \L \mapsto \sigma_{\L}
\end{equation}
 such that for every $b\in B$ we have 
$ 
\sigma_{\L}(b)= \L_{|X_b}.
$ Observe that the map (\ref{brauer}) may very well fail to be surjective; 
see \cite[Ch. 8, Prop. 4]{BLR}.
\end{remark}
Let $\phi:\X \to B$ be a  family of curves 
as above, and let $b_0\in B$ be a fixed (closed) point.
As usual, we set $X_0=\phi^{-1}(b_0)$;
we assume that $X_0$ is a nodal curve 
 and denote by $\G$ the dual graph of $X_0$.
We 
identify $\Div(\G)=\Z^{V(\G)}$, so that we have a map
$$
\Pic (X_0) \la \Div(\G)=\Z^{V(\G)};\quad L\mapsto \mdeg L.
$$
Now, we have a {\it specialization} map  $\tau=\tau_{\phi, b_0}$
mapping a section of $\pi$ to the multidegree of its value on $b_0$:
\begin{equation}
\label{tau}
\begin{array}{lccr}
\Picphi(B) &\stackrel{\tau}{\la} &\Div(\G) \\
\\
  \  \  \  \sigma&\mapsto &\mdeg \ \sigma(b_0).
\end{array}
\end{equation}
\begin{remark}
To connect with Baker's work we need to temporarily drop the general conventions  stated at the begining of Section~\ref{agsec}.
The definition of the map $\tau$ above is
inspired by the specialization map, denoted by $\rho$, defined   in 
\cite[Subsection 2.1]{bakersp}. Our definition is a slight generalization:
the map $\rho$   coincides with the composition of our $\tau$
with the canonical map (\ref{brauer}) (defined at the level of divisors, rather than linear equivalence classes). More precisely: let $B=\Spec R$ with $R$ a complete DVR with algebraically closed residue field (which is the set-up of \cite{bakersp}, to which our previous definitions are easily seen to extend), then
 $\rho$ can be defined as follows
$$
\rho:\Div(\X)\la \Pic (\X) \la \Picphi (B)\stackrel{\tau}{\la} \Div(\G).
$$
So, we use the terminology ``specialization map" for consistency with \cite{bakersp},
since no confusion should arise.
\end{remark}
 \begin{remark}
\label{unbounded}
  Let $\phi$ be a regular one-parameter smoothing  of a reducible nodal curve 
$X_0$, let
 $\G $ be  the dual graph of $X_0$; pick $\L\in \Pic (\X)$.
By the classical upper-semicontinuity theorem we have, for all $b\in B$,
\begin{equation}
\label{upper}
r(X_b,  \L_{X_b}) \leq r(X_0, {\L}_{X_0}).
\end{equation}
Let $d$ be the $\phi$-relative degree of  $\L$, i.e. $d=\deg  \L_{X_b}$ for all $b\in B$.
Now, denote by $\L^*$ the restriction of $\L$ away from $X_0$.
Then there are infinitely many different completions of $\L^*$  to a line bundle on the whole of $\X$,
and hence 
infinitely many line bundles on $X_0$ appearing as specializations of  $\L^*$.
The point  is that, as $\L^*$ is fixed, the term on the right of the inequality (\ref{upper}) is unbounded.
Indeed, let $C\subset X_0$ be an irreducible component, and set $d_C:=\deg_C\L$.
Then for every $n\in \Z$ the line bundle 
$$\L^{(n)}:=\L(-nC)\in \Pic(\X)
$$ is a completion of $\L^*$, having $\phi$-relative degree $d$, just like $\L$,  of course. If $n\geq 1$ then 
$$
\deg_C\L^{(n)}=d_C-nC\cdot C\geq d_C+n,
$$
and if $n\gg 0$
the degree of  $\L^{(n)}$ on $\overline{X_0\smallsetminus C}$ is negative. Also, if $n\gg 0$ we have
$$
H^0(X_0,\L^{(n)}_{X_0})\supset H^0(C,\L^{(n)}_{C}(-C\cap \overline{X_0\smallsetminus C}))
$$
and the dimension of the space on the right goes to $+\infty$ as $n$ grows. Therefore $r(X_0,\L^{(n)}_{X_0})$ goes to $+\infty$ with $n$.

We now  look at  the corresponding situation for the combinatorial ranks.
Consider the  specialization  of $\L^{(n)}$  via the map  $\tau$  defined in (\ref{tau}):
$$
D^{(n)}:=\tau (\sigma_{\L^{(n)}})\in \Div (\G).
$$
Then   for every $n,m\in \Z$ we have $D^{(n)}\sim D^{(m)}$ hence, by Remark~\ref{basic}
$$
r_{\G}(D^{(n)})=r_{\G}(D^{(m)})\leq d. 
$$
Concluding,  the combinatorial rank   behaves better under specialization, than the algebro-geometric rank, as it depends only on $\L^*$ (and not on the choice of   completion) and it is bounded by the relative degree of $\L^*$.
 \end{remark}

  The following is a refinement of Baker's Specialization Lemma, which we like to view as a mixed  upper semicontinuity result.
  
\begin{prop}[Mixed semicontinuity]
\label{spe}
Let $\phi:\X \to B$ be a regular one-parameter smoothing  of a nodal curve $X_0$  having dual graph $\G$; consider the map $\tau$ defined in (\ref{tau}).
Then for every $\sigma\in \Picphi(B)$ there exists an open neighborhood $U\subset B$ of $b_0$ such that for every $b\in U$  with  $b\neq b_0$
\begin{equation}
\label{mixed}
r(X_b, \sigma (b))\leq r_{\G}(\tau (\sigma)).
\end{equation}
\end{prop}
\begin{remark}
A remarkable special case is that of an element in $\Picphi(B)$ of type $\sigma_{\L}$,
for some $\L\in \Pic \X$,
as  defined in (\ref{brauer}).
Then   \ref{spe} states that
 there is a neighborhood $U\subset B$ of $b_0$ such that 
$$
h^0(X_b, \L_{|X_b})-1\leq r_{\G}(\mdeg   \L_{|X_0})
$$
for every $b\in U$ with $b\neq b_0$.
\end{remark}
\begin{remark}
\label{glue}
By Remark~\ref{unbounded} it is clear that
the assumption $b\neq b_0$ is necessary, as $r(X_0, \sigma (b_0))$ is unbounded
(if  $X_0$
is reducible), whereas,
 if $d$ denotes the   $\phi$-relative degree of $\sigma$,   then $r_{\G}(\tau (\sigma))\leq d$.

\end{remark}
\noindent
{\it{Proof of Proposition~\ref{spe}.}}
We begin with the following
\begin{claim}
We can work up to finite \'etale base change.
\end{claim}
To prove the claim, let $\e:B'\to B$ be a  finite \'etale morphism and let
\begin{equation}\label{diag1}
\xymatrix{
\X' \ar@{^{}->}[r]^{\hat{\e}}  \ar@{->}[d]_{\phi '} & \X\ar@{->}[d]^\phi \\
B' \ar@{->}^\e[r] &B
}
\end{equation}
be the corresponding base change,
so that $\hat{\e}:\X'=\X\times_BB'\to \X$  is the projection. 
As $\e$ is \'etale, the total space $\X'$ is nonsingular.
Let $b'_0\in B'$ be such that $\e(b'_0)=b_0$;
of course, the preimage of $\phi'$ over $b'_0$ is isomorphic to $X_0$, and hence it has the same dual graph $\G$. Now,
  let $\tau'=\tau_{\phi',b_0'}$ be the specialization map of $\phi'$  with respect to $b'_0$
(see  (\ref{tau})). 
We have a commutative diagram
\begin{equation}\label{diag2}
\xymatrix{
\Picphi(B) \ar@{^{}->}^{\e^*}[r]  \ar@{->}[d]_{\tau } &\Pic_{\phi '}(B')\ar@{->}[d]^{\tau '} \\
\Div(\G)\ar@{=}[r] &\Div(\G)
}
\end{equation}
where $\e^*$ is the pull-back of sections
(for $\sigma \in \Picphi(B)$ and $b'\in B'$ we have
 $$\e^*(\sigma)(b'):=\hat{\e}^*\sigma(\e(b'))\in \Pic X'_{b'}$$ 
 where $\hat{\e}^*: \Pic X_{\e(b')}\to  \Pic X'_{b'}$ is the ordinary pull-back.)

Up to replacing $B'$ with an open neighborhood of $b'_0$ we can assume that
$\phi' $ has smooth fibers away from $b'_0$, so that $\phi'$ satisfies the same hypotheses as $\phi$.
Now assume the result holds for $\phi'$. Let $\sigma\in \Picphi(B)$ and let 
$$
\sigma'=\e^*(\sigma).
$$
Let $U'\subset B'$ be a neighborhood of $b'_0$ such that
\begin{equation}
\label{th'}
r(X'_{b'}, \sigma' (b'))\leq r_{\G}(\tau' (\sigma'))
\end{equation}
for every $b'\in U'\smallsetminus \{b_0'\}$.
Pick an open neighborhood $U\subset \e (U')$ of $b_0$; for every $b\in U$ 
and every $b' \in \epsilon^{-1}(b)$
we have $X_b\cong X'_{b'}$
and $r(X_{b}, \sigma  (b))=r(X'_{b'}, \sigma' (b'))$. On the other hand the commutativity of (\ref{diag2})
gives
$\tau (\sigma)=\tau'(\e^*(\sigma))=\tau'(\sigma ')$. Combining with (\ref{th'}) we get that (\ref{mixed}) holds. The claim is proved.

\

By the claim, we can assume that $\phi$ has   sections;
in particular we shall assume that
for every irreducible component $C_v\subset X_0$, the map  $\phi$ has a section $s_v$  
intersecting $C_v$.

 From now on we shall work up to replacing $B$ by an open subset containing $b_0$, which we can obviously do.

Fix $\sigma\in \Picphi(B)$.
The existence of a section of $\phi$ ensures that the canonical map $\Pic \X \la \Picphi(B)$ introduced in (\ref{brauer})
is surjective; see \cite[Ch. 8, Prop. 4]{BLR}.
Hence there exists $\L\in \Pic \X$ such that
for every $b\in B$ we have $\L_{|X_b}=\sigma (b)$.
We write $L_b=\L_{|X_b}$ and $L_0=\L_{|X_0}$.

We shall prove that if $r(X_b, L_b)\geq r$ then
$ r_{\G}(\mdeg L_0)\geq r$ (which is of course equivalent to (\ref{mixed})).

If $r=-1$ there is nothing to prove; so we may assume $r\geq 0$. 

Since $r\geq 0$ we have that $\phi_* \L$ has positive rank ($\phi$ is proper). There  exists 
an $M\in \Pic B$ such that $h^0(B, \phi_* \L\otimes M)\geq 1$.
Therefore, as $\phi_* \phi^*M=M$,
$$
h^0(\X,  \L\otimes \phi^*M)=h^0(B, \phi_* \L\otimes M)\geq 1.
$$
Hence there exists an effective divisor $E$ on $\X$ such that $\L=\O_{\X}(E)$.
We can decompose
$$
E=E_{\rm{hor}}+E_{\rm{ver}}, 
$$
with $E_{\rm{hor}}$ an effective ``horizontal" divisor  (i.e. the support of $E_{\rm{hor}}$ contains no component of the fibers of $\phi$) and $E_{\rm{ver}}$ an effective ``vertical" divisor (i.e. $E_{\rm{ver}}$ is entirely supported on fibers of $\phi$). Up to shrinking $B$ we can further assume that $E_{\rm{ver}}$ is supported only on $X_0$.

Now we are ready to finish the proof of our result, using induction on $r$.
What now follows is similar to Baker's proof of his Specialization Lemma.
If $r=0$ we need to show that $\mdeg L_0$ is equivalent to an effective divisor on $\G$.
We have
$$
\mdeg L_0=\mdeg_{X_0} E_{\rm{hor}}+\mdeg_{X_0} E_{\rm{ver}}. 
$$
Now, 
$$\mdeg_{X_0} E_{\rm{hor}}\geq 0$$ because $E_{\rm{hor}}$ is a horizontal effective divisor.
On the other hand, 
by hypothesis, $E_{\rm{ver}}$ is of type 
$$E_{\rm{ver}}=\sum_{v\in V(\G)} n_vC_v$$ where $X_0=\cup_{v\in V(\G)} C_v$.

Therefore, using Remark~\ref{twisterrk} and (\ref{twistdeg}) , we have
$$
\mdeg_{X_0} E_{\rm{ver}}=\sum _{v\in V(\G)} n_vT_v\in \Prin (\G).
$$
In conclusion, $\mdeg L_0\sim \mdeg_{X_0} E_{\rm{hor}}$, so we are done. The case $r=0$ is finished.

Assume $r\geq 1$. For every vertex $v\in \G$, let $s_v$ be the previously introduced section of $\phi$ passing through $C_v$,
and let
$A_v:=s_v(B)$ be the corresponding effective divisor on $\X$.
We have, of course,
$$
r(X_b, \L(-A_v)_{|X_b})\geq r(X_b,  L_b)-1 \geq r-1
$$
for all $b\in B$. 
Denote by
$$D_0=\mdeg L_0\in \Div (\G).$$
Let $$
\hat{\tau}:\Pic \X \la \Picphi(B)\stackrel{\tau}{\la} \Div (\G)
$$
be the composition of  $\tau$ with
the canonical map (\ref{brauer});  we have 
$$
\hat{\tau}( \L(-A_v))=D_0-v \in \Div (\G).
$$
By the induction hypothesis,
we have
$$
r_{\G}(D_0-v)\geq r-1.
$$
This holds for every $v\in V(\G)$. Therefore, using \cite[Lemma 2.7]{bakersp} we get  $r_{\G}(D_0)\geq r$
and we are done. \qed

\section{Specialization for   graphs   with loops}
\label{loopsec}

Let $\G$ be a graph admitting some loops.
The definition of rank of a divisor given before is independent of the loops, and in fact it is not 
a satisfactory one; for example, it trivially violates the Riemann-Roch formula
which does hold on graphs free from loops, by \cite{BN}.
We shall now give a better definition, and extend the Specialization Lemma to this definition. 
First, we shall introduce some useful terminology.

\begin{defi}
Let $\G$ be a graph. A {\it refinement} of $\G$ is a graph $\hG$ obtained by inserting a (finite) set of vertices in the interior of the edges of $\G$.
We have a natural inclusion 
$V(\G)\subset V(\hG)$ which induces an injective group homomorphism
\begin{equation}
\label{iota}
\iota_{\G,\hG}:\Div \G \ha \Div \hG.
\end{equation}
\end{defi}
We shall often write simply $\iota=\iota_{\G,\hG}$.

In general, the map (\ref{iota}) is not 
compatible with linear equivalence, nor  does it
preserve   the rank (see Example~\ref{loop1}). There is, however, a useful  situation in which the rank is preserved.
\begin{remark}
\label{HKN}
Let $\G$ be a graph and let $\G^{(n)}$ be the refinement of $\G$ given by inserting $n$ vertices in the interior of every edge of $\G$. Then, if $\G$ has no loops, for every $D\in \Div (\G)$ we have
$$
r_{\G}(D)=r_{\G^{(n)}}(\iota_{\G, \G^{(n)}}(D)).
$$
 This follows from \cite[Corollary 22]{HKN}; see also \cite[Thm 1.3]{luo}.
\end{remark}

There is another type of refinement which preserves the rank, and which enables us to define the rank for a graph with loops in a sharper way. So,
let $\G$ be a graph admitting $l$ loops;
denote by $\{\ell_1,\ldots,\ell_l\}\subset E(\G)$  the set of loop-edges.

 Let $n_1,\ldots, n_l$ be positive integers and let $\underline{n}:=(n_1,\ldots, n_l)$
be their ordered sequence.
Let  
 $\G^{\underline{n}}$
be the graph obtained by inserting  $n_i$   vertices in the interior of every loop-edge $\ell_i$, and leaving the other edges untouched. 
As $\G^{\underline{n}}$ is a refinement of $\G$ there is a natural   map
$  
\iota :\Div \G\to \Div \G^{\underline{n}}.
$  
For every $D\in \Div(\G)$
we define
\begin{equation}
\label{rs}
\rs_{\G}(D) :=r_{\G^{\underline{n}}}(\iota (D)).
\end{equation}

\begin{remark}
\label{invloop}
The above definition is independent on the choice of the numbers $n_1,\ldots, n_l$,
provided $n_i\geq 1$ for all $i=1,\ldots, l$.
More precisely,
for every  $\underline{m}:=(m_1,\ldots, m_l)$  with $m_i\geq 1$ for every $i=1,\ldots, m$,
from \cite[Theorems 1.3 and 1.5]{luo}  we obtain
$$
 r_{\G^{\underline{n}}}(\iota_{\G,\G^{\underline{n}}} (D))=r_{\G^{\underline{m}}}(\iota_{\G,\G^{\underline{m}}} (D)).
 $$
\end{remark}
With definition (\ref{rs}) the Riemann-Roch formula holds; see \cite{ABC}.
  
Comparing with the rank $r_{\G}$ defined earlier, it is not hard to see that 
$$
\rs_{\G}(D) \leq r_{\G}(D).  $$ 
\begin{example}
\label{loop1}
Consider the graph $\G$ of genus 2 drawn in the picture below, so
$\G$ has one loop-edge attached to the vertex $v$. To compute $\rs_{\G}$ we can use the graph
${\Gamma}^{(1)} $ obtained by inserting one vertex $u$ in the loop-edge of $\G$.
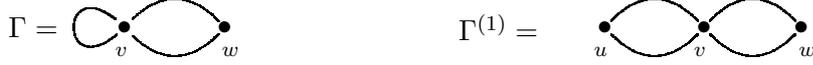
\begin{figure}[h]
\begin{equation*}
\xymatrix@=.5pc{
&&&&&&&&&&&&&&\\
\Gamma = &&\ar@{-}@(ul,dl)*{\bullet}\ar @{-} @/_.9pc/[rrr]_(.1)v_(.9)w  \ar@{-} @/^.9pc/[rrr]
&&& *{\bullet} &&&&&&&
{\Gamma}^{(1)} = &&*{\bullet} \ar @{-} @/_.9pc/[rrr]_(.1)u \ar@{-} @/^.9pc/[rrr]
&&& *{\bullet}\ar @{-} @/_.9pc/[rrr]_(.1)v_(.9)w \ar@{-} @/^.9pc/[rrr]
&&& *{\bullet} &&&\\
&&&&&&&&&&&&&&\\
}
\end{equation*}
\end{figure}

Consider the divisor $v+w\in \Div\G$. Then one easily checks that
$$
r_{\G}(v+w)=1,\quad\quad \rs_{\G}(v+w)=0.
$$
On the other hand
$
r_{\G}(2v)=1=\rs_{\G}(2v).
$

\end{example}

We now  prove that  Proposition~\ref{spe} holds with this   definition of rank.
\begin{prop}
\label{speloop}
Assume that $\phi:\X \to B$   satisfies the same assumptions as Proposition~\ref{spe}.
Then for every $\sigma\in \Picphi(B)$ there exists an open neighborhood $U\subset B$ of $b_0$ such that for every $b\in U\smallsetminus  b_0$ we have
\begin{equation}
\label{mixedl}
r(X_b, \sigma (b))\leq \rs_{\G}(\tau (\sigma)).
\end{equation}

\end{prop}
\begin{proof}
Let $l$ be the number of loops of $\G$.
As usual, we shall identify the edges of $\G$ with the nodes of $X_0$.

Let $\widehat{\X}\to \X$ be the blow up at all the $l$ loops of $\G$.
By the regularity assumption on $\X$, every exceptional divisor  of this blow-up is a
$(-1)$-curve  of $\widehat{\X}$, appearing with multiplicity $2$ in the fiber  which contains it.
Hence the family of curves $\widehat{\phi}:\widehat{\X}\to B$
has   non-reduced fiber over $b_0$.

We now apply the same construction as 
\cite[Sect. III.9]{BPV} to obtain a family of nodal curves with regular total space.
Let $B^1\to B$ be a degree-2 covering ramified over $b_0$ and let 
$\widehat{\X}^1\to B^1$ be the base change of $\widehat{\phi}$.
Let $\Y\to \widehat{\X}^1$ be the normalization, so that we have a family
$\Y \to B^1$ of curves with nodal special fiber.
The dual graph of the special fiber of $\Y \to B^1$ is the refinement $\hG$ of
$\G$ obtained by adding one vertex in the interior of every loop edge of $\G$.
For further use, we denote by $\{e_i^+, e_i^-, i=1,\ldots,l\}\subset E(\hG)$ these new edges  replacing the loops
of $\G$.

Now, $\Y$ has an $A_1$-singularity at every  edge of $\hG$ other than the
$2l$ edges $\{e_i^+, e_i^-\}$, and no other singularity
(\cite[proof of Prop.(III.9.2)]{ BPV}).
Let 
$\ZZ\to \Y$ be the resolution of  all such $A_1$ singularities; now $\ZZ$ is a regular surface
and $\ZZ\to B^1$ is a regular smoothing of its special fiber.

Denote by $\G_Z$ the dual graph  of the special fiber of $\ZZ$;
then $\G_Z$ is a refinement of $\hG$ and hence of $\G$.
By construction, the edges $e_i^+, e_i^- \in E(\hG)$ are not refined in $\G_Z$, so they correspond to edges of $\G_Z$. We 
abuse notation setting
$$
\{e_i^+, e_i^-, \  i=1,\ldots,l\}\subset E(\G_Z).
$$
Finally, we introduce the refinement, $\hG^{(1)}$ of $\hG$ obtained by adding one vertex in the interior of every
edge; of course, $\hG^{(1)}$  is also a refinement of $\G_Z$.
We have a sequence of    maps
$$
\Div \G\stackrel{\iota }{\la} \Div \hG \stackrel{}{\la} \Div \G_Z \la \Div \hG^{(1)}.
$$
The following picture helps to  keep track of the above set-up.
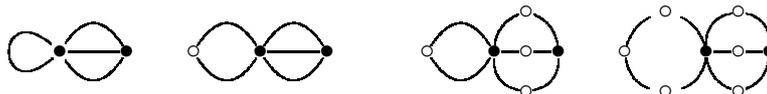
\begin{figure}[h]
\begin{equation*}
\xymatrix@=.5pc{
 &&&&&&&&&&&&&&&&*{\circ}\ar@{-}@/^.3pc/[dr]\ar@{-}@/_.3pc/[dl]
  &&&&
 \circ\ar@{-}@/_.3pc/[dl]\ar@{-}@/^.3pc/[dr]
&&
 *{\circ}\ar@{-}@/^.4pc/[dr]\ar@{-}@/_.4pc/[dl]
  \\
&&\ar@{-}@(ul,dl)*{\bullet}\ar@{-}[rr]\ar @{-} @/_.9pc/[rr] \ar@{-} @/^.9pc/[rr]
&&*{\bullet} 
&&
*{\circ} \ar @{-} @/_.9pc/[rr] \ar@{-} @/^.9pc/[rr]
&&*{\bullet}\ar@{-}[rr]\ar @{-} @/_.9pc/[rr] \ar@{-} @/^.9pc/[rr]
&& *{\bullet} 
 &&& 
*{\circ} \ar@{-} @/_.9pc/[rr] \ar@{-} @/^.9pc/[rr]
&&*{\bullet}\ar@{-}[r]&*{\circ}\ar@{-}[r]& *{\bullet}
&&
*{\circ}&&*{\bullet}\ar@{-}[r]&*{\circ}\ar@{-}[r]& *{\bullet} \\
&&&&&&&&&&&&&&&&*{\circ}\ar@{-}@/^.4pc/[ul]\ar@{-}@/_.4pc/[ur]
&&&&\circ\ar@{-}@/^.4pc/[ul]\ar@{-}@/_.4pc/[ur]
 &&
 *{\circ}\ar@{-}@/^.3pc/[ul]\ar@{-}@/_.3pc/[ur]\\
}
\end{equation*}
\caption{Refinements  from left to right: $ \Gamma, \widehat{\Gamma}, \Gamma_Z, \widehat{\Gamma}^{(1)}$. The original vertices are drawn as $\bullet$, the others as $\circ$.}
\end{figure}

Let $D\in \Div \G.$
We have  
$$
\rs_{\G}(D)=r_{\hG}(\iota (D))=r_{\hG^{(1)}}({\iota_{\hG,\hG^{(1)}}}(\iota (D)))=
r_{\G_Z}(\iota_{\hG,\G_Z}(\iota (D)))=r_{\G_Z}(\iota_{\G,\G_Z}(D)).
$$
Indeed, the first ``$=$" is the definition of $\rs$; the second follows from Remark~\ref{HKN}
(the refinement $\hG^{(1)}$ of $\hG$ adds one vertex in the interior of every
edge); 
  the third is the  invariance  mentioned in \ref{invloop}
  ($\hG^{(1)}$ adds one vertex in the interior of every edge $e_i^+$ and $e_i^-$
  of $\G_Z$);
  the last   is  $\iota_{\hG,\G_Z} \circ \iota =
\iota_{\G,\G_Z} $.

Therefore we have reduced our statement to the loop-free situation for the family
$\ZZ\to B^1$, which has been proved in Proposition ~\ref{spe}.
\end{proof}

\section{On the emptyness of  Brill-Noether loci}
\label{BNsec}
From now on we
shall assume $g\geq 2$, as, by the Riemann-Roch Theorem, 
 the content of this section is  interesting only in this case.
  For an algebraic curve $C$ one defines the ``Brill-Noether variety" of $C$  as 
follows: $W^r_d(C)=\{L\in \Pic^dC: r(C,L)\geq r\}$.

Consider the Brill-Noether number
$ 
\rho^r_d(g):=g-(r+1)(g -d+r).
$  
We   recall two fundamental theorems about  $W^r_d(C)$.
The ``Existence Theorem"
due to Kempf \cite{kempf} and Kleiman-Laksov,  \cite{KL1} \cite{KL2}; see also \cite[Thm. (1.1) Chapt. V]{ACGH}:
\begin{nota}{\bf Existence Theorem.}
\label{Existence}
 {\it If $\rho^r_d(g)\geq  0$ then for   every smooth projective curve $C$ of genus $g$
we have
  $W^r_d(C)\neq \emptyset$. Moreover, if $r\geq d-g$, then every irreducible component of
  $W^r_d(C)$ has dimension at least $\rho^r_d(g)$.}
  \end{nota}

The assumption $r\geq d-g$   above   and in Theorem~\ref{BNthm} below is needed simply because if $r< d-g$ then, by Riemann-Roch,
$W^r_d(C)=\Pic^dC$ and hence $\dim W^r_d(C) =g<\rho^r_d(g)$.

\

Next is  the ``Brill-Noether  Theorem"
proved by Griffiths-Harris \cite{GH}; see \cite[Thm. (1.5) Chapt. V]{ACGH}:
\begin{nota}{\bf Brill-Noether    Theorem.}
\label{BNthm}
 {\it If $\rho^r_d(g)< 0$ then for a general smooth projective curve $C$ of genus $g$
we have
  $W^r_d(C)= \emptyset$. Moreover if $r\geq d-g$ then every irreducible component of
  $W^r_d(C)$ has dimension equal to $\rho^r_d(g)$.}
  \end{nota}
 The word ``general" above means:
for every $C$ in a nonempty Zariski open subset 
of the moduli space of smooth curves of genus $g$.

We now investigate whether analogous results hold for graphs.
Let $\G$ be a graph of genus $g$; set
$$
W^r_d(\G):=\{[D]\in \Pic^d \G: \  \rs_{\G}(D)\geq r\}.
$$
The following Theorem~\ref{BN+} proves conjecture 
\cite[Conj. 3.9 (1)]{bakersp} (and also \cite[Conj. 3.10 (1)]{bakersp}).

\begin{thm}[Existence Theorem for graphs]
\label{BN+}
Let $g,d,r$ be integers such that $\rho^r_d(g)\geq 0$.
Then for every graph
$\G$  of genus $g$ we have $W^r_d(\G)\neq \emptyset$.
\end{thm}

\begin{proof}
For later use, observe that we will prove that if $r,d,g$ are such that $W^r_d(C)$ is non-empty for a general smooth curve of genus $g$,
then $W^r_d(\G)\neq \emptyset$ for every graph $\G$ of genus $g$.

Notice also that if $r<d-g$ the result is trivial, by Riemann-Roch.

Let $X_0$ be  a general nodal curve whose dual graph is $\G$. 
Fix a regular  one-parameter 
smoothing $\phi:\X \to B$ of $X_0$;
 the existence of such a one-parameter smoothing is well known: it follows, for example, from
 \cite[Prop 1.5, pp. 81,82]{DM}. Moreover,  a general one-parameter smoothing of $X_0$
 will be a regular one.
Denote by $b_0$ the special point of $B$;
we can work up to replacing $B$ with an open neighborhood of $b_0$.
Furthermore, we choose $\phi$ so that it has a section (which we can do up to \'etale base change).

By the   Existence Theorem~\ref{Existence} for curves, 
 the assumption $\rho^r_d(g)\geq 0$ implies that for every curve $C$ of genus $g$
we have
  $W^r_d(C)\neq \emptyset$.

Now, by  \cite[Sect. 2]{AC}
(or   \cite[Ch. 21, Sect. 3]{gac}),
for any family of smooth projective curves $\psi{:\mathcal C}\to B$    admitting    a section,
there exists a $B$-scheme $ W^r_{d,\psi}\to B$ 
whose fiber over every $b\in B$   is
 $W^r_d(C_b)$.
Moreover there is a natural injective morphism of $B$-schemes,
$$W^r_{d,\psi}\ha  \Pic_{\psi}
$$ which we view as an inclusion. We want to use  this construction for our one-parameter smoothing $\phi$ of   $X_0$, but $\phi$ admits a singular fiber.
Since the restriction $\phi^*:\X^*\to B^*=B\smallsetminus \{b_0\}$ is a family of smooth curves, the above set-up gives the relative Brill-Noether variety
$W^r_{d,\phi^*}\to B^*$. We let 
$\overline{W^r_{d,\phi}}\to B$ be the closure of $W^r_{d,\phi^*}$ in 
the compactified Picard scheme 
$\overline{P_{\phi}^d}\to B$; 
recall that the restriction of $\overline{P_{\phi}^d}$ over $B^*$ coincides with $\Pic^d_{\phi ^*}$,
and
  every point of its fiber  
over $b_0$   corresponds to
a line bundle on a partial normalization of $X_0$;
see \cite{cner} for details. 
We define $W^r_{d,\phi}$ as the intersection of 
$\overline{W^r_{d,\phi}}$ with $\Pic^d_{\phi}$.
By the above discussion, one obtains that this $B$-scheme $W^r_{d,\phi}\to B$
has non empty fiber over $b_0$, and, by the upper-semicontinuity of $h^0$,
this fiber   is contained in $W^r_d(X_0)$.
We simplify the notation and write
$$
W_{\phi}:=W^r_{d,\phi}\la B.
$$
There exists a finite  covering $\delta:B^1\to B$, totally ramified over $b_0$,
such that the base change
 $W_{\phi}\times _BB^1\la B^1$ 
admits a section. 
Now, denote by $\phi^1:\X^1\to B^1$ the base change of $\phi$; notice that the dual graph of the special fiber of $\phi^1$ is again $\G$.
We have
$$W_{\phi^1}=W_{\phi}\times _BB^1\la B^1,$$
as compatibility with base change holds; see \cite[Ch. 21, Sect. 3]{gac}.
We denote by $\sigma:B^1\to W_{\phi^1}$ the section   mentioned above.

The family of curves $\phi^1$ is no longer a regular smoothing of its special fiber;
the situation we are about to describe is detailed in \cite[Sect. III.9]{BPV}.
We denote by $\widetilde{\X^1}\to \X^1$ the normalization of $\X^1$
and by 
$\Y\to \widetilde{X^1}$ its resolution of singularities.
More precisely, if,  locally at $b_0$, the covering $\delta$ has the form $t\mapsto t^{n+1}$,
then the surface  $\widetilde{\X^1}$ has a singular point of type $A_{n}$ at every node of its special fiber. Therefore the map $\Y\to \widetilde{X^1}$ replaces every node of the special fiber by a chain of
$n$ exceptional components.
Denote by
$$
\chi:\Y \la B^1
$$
the family over $B^1$ obtained by composing $\Y\to \widetilde{X^1}\to \X^1\to B^1$.
Now $\chi$ is a regular smoothing of its special fiber,   $Y_0$.
The dual graph of $Y_0$ is obtained from the dual graph, $\G$, of the special fiber of $\phi^1$,
by inserting $n$ vertices in the interior of every edge.
Hence we denote by $\G^{(n)}$ the dual graph of $Y_0$.
We have a natural   map
$$
\iota=\iota_{\Gamma, \G^{(n)}}:\Div(\G)\la \Div(\G^{(n)})
$$
and this map preserves the rank, i.e. $\rs_{\G}(D)=\rs_{\G^{(n)}}(\iota (D))$
for every $D\in \Div(\G)$, by Remark~\ref{HKN}.
Now, we have a commutative diagram
\begin{equation}\label{diag1}
\xymatrix{
\Y \ar@{^{}->}[r]^{\beta}  \ar@{->}[d]_{\chi} & \X^1\ar@{->}[d]^{\phi^1} \\
B^1\ar@{=}[r]   &B^1
}
\end{equation}
and an associated $B^1$-map $\Pic_{\phi^1}\to \Pic_{\chi}$, hence
also a
  map
$$
\beta^*:\Pic_{\phi^1}(B^1)\la \Pic_{\chi}(B^1).
$$
The previously defined section  $\sigma:B^1\to W_{\phi^1}$ is an element of $\Pic_{\phi^1}(B^1)$,
so that $\beta^*(\sigma)\in \Pic_{\chi}(B^1)$.
By  Proposition~\ref{speloop} applied to $\chi$ and $\beta^*(\sigma)$, we have for every $b\in B^1$, $b\neq \delta^{-1}(b_0)$
$$
r\leq r\Bigr(Y_b, \beta^*(\sigma) (b)\Bigl)\leq \rs_{\G^{(n)}}\Bigr(\tau \bigr(\beta^*(\sigma)\bigl)\Bigl).
$$
On the other hand, by construction,
the divisor $\tau (\beta^*(\sigma))\in \Div (\G^{(n)})$ corresponds to
 $\tau(\sigma)\in \Div (\G)$ under the refinement map
 $\iota:\Div \G \to \Div \G^{(n)}$;
 in symbols 
$$
\tau (\beta^*(\sigma))=\iota(\tau(\sigma)).
$$
Since $\iota$ preserves the ranks, we obtain
$$
r\leq \rs_{\G^{(n)}}(\tau (\beta^*(\sigma)))= \rs_{\G^{(n)}}(\iota(\tau(\sigma)))=\rs_{\G}(\tau(\sigma)). 
$$
Hence we have $\tau(\sigma)\in W^r_d(\G)$. We have thus proved that $W^r_d(\G)$ is not empty, so we are done.
\end{proof}

Although the previous theorem is purely graph theoretic, our proof
uses algebraic geometry. So we wish to propose the following problem.

\begin{prob}
Find a purely combinatorial proof for Theorem~\ref{BN+}.
\end{prob}

We now turn to the Brill-Noether Theorem \ref{BNthm}. There is an important difference with the Existence Theorem, namely the  Brill-Noether Theorem is valid for a  general curve, and is well known to fail 
for some particular curves (for example, hyperelliptic curves of genus at least 3).
This difference reflects itself in the subsequent discussion.

\begin{fact}
\label{fact}
(\cite[Conjecture 3.9 (2)]{bakersp} - \cite[Theorem 1.1]{CDPR})
Assume     $\rho^r_d(g)<0$.
 There exists a graph of $\G$ of genus $g$ such that $W^r_d(\G)=\emptyset$.
 \end{fact} 
\begin{remark}
Theorem 1.1 of \cite{CDPR} is actually a stronger result, from which the above fact follows.
In particular, the authors obtain the following. Let $\G$ be a chain of $g$ cycles
$\Delta _1,\ldots \Delta _g$ each of which has $2g-1$ 
cyclically ordered vertices $V(\Delta_i)=\{v_1^i,\ldots, v_{2g-1}^i\}$
and such that
$$
v_{2g-1}^1=v_1^2,\quad v_{2g-1}^2=v_1^3,\  \ldots, \  v_{2g-1}^i=v_1^{i+1},\  \ldots, \  v_{2g-1}^{g-1}=v_1^g,
$$
with no other identifications. Then $W^r_d(\G)=\emptyset$ if $\rho^r_d(g)<0$.
\end{remark}
Recall that a graph  is called $3$-{\it regular} if all of its vertices have valency 3, and that a 3-regular graph
is 3-connected if  and only if it is 3-edge connected.
\begin{conj}
\label{conj}
Assume $g\geq 2$ and   $\rho^r_d(g)<0$. 
\begin{enumerate}
 \item
 \label{conj1}
There exists   a $3$-regular    graph $\Gamma$ of genus $g$ 
for which  $W^r_d(\G)=\emptyset$.
\item
 \label{conj2}
Let $\G$ be a graph of genus $g$ with the highest number of automorphisms (among graphs of genus $g$). Then $W^r_d(\G)=\emptyset$.
\end{enumerate}
\end{conj}

\begin{remark}
These two conjectures are quite different, but they are both inspired   by the analogies between the moduli space $\Mgb$ of Deligne-Mumford stable curves and the moduli space $\Mgt$ of tropical curves
(see \cite{BMV} and \cite{Ctrop}).  

Both $\Mgb$ and $\Mgt$ admit a partition into strata 
which are indexed by graphs (the dual graphs of stable curves for $\Mgb$, the underlying graphs 
of tropical curves for $\Mgt$); in both cases, the set of  strata
is partially ordered 
under inclusion of closures.
Recall that
the generic points of $\Mgt$,  i.e. the points in the top dimensional strata, 
parametrize  tropical curves whose underlying graph is $3$-regular.
As we said,   for a general point of $\Mgb$, i.e. for a general smooth  curve of genus $g$, the Brill-Noether variety is empty whenever $\rho$ is negative. 
By analogy, we  may ask whether some of the   generic (from the tropical point of view) graphs,
i.e. the 3-regular graphs, have an empty Brill-Noether locus when  $\rho$ is negative. 
This explains part (\ref{conj1}).
Finally, notice that there do exist 3-regular graphs that are not Brill-Noether general (see the next remark); hence if (\ref{conj1}) holds, it would be interesting to characterize the graphs that satisfy it.

Next, to motivate the second part, we recall that by \cite[Thm. 4.7]{Ctrop}, the 
natural bijection between the partitions of $\Mgb$ and  
$\Mgt$ 
is order reversing.
This suggests that a general point of $\Mgb$  corresponds to a special point
of $\Mgt$. Let us focus on the top dimensional strata: smooth curves on   one side,
3-regular graphs on the other side.
It is well known that
a general smooth curve  in $\Mgb$ has no nontrivial automorphisms, hence, with the above  ``reversion" phenomenon in mind,
we may think of a $3$-regular graph  with the greatest number of symmetries as its
analog.   Since, as we said,   a general curve  has  empty Brill-Noether locus when  $\rho$ is negative
we can ask whether the same holds for the analogous graphs.
\end{remark}
\begin{remark}
 Suppose $g\geq 3$.  
A  $3$-regular and  $3$-connected graphs is known   to be non-hyperelliptic.
Morover, for every $g$ there exists a $3$-regular hyperelliptic graph  of connectivity $2$ (such graphs are, of course, not Brill-Noether general).
\end{remark}

In the next proposition, the new fact with respect to the tropical proof of the Brill-Noether Theorem
 of \cite{CDPR}  is   that it suffices to have an ordinary graph,
rather than a tropical curve,   for which the Brill-Noether
locus is empty (and checking the emptyness of $W^r_d(\G)$ for a graph $\G$
is a finite amount of work).

\begin{prop}
 \label{BNBN}
Suppose that for some integers 
$d,g,r$ there exists a graph $\G$ of genus $g$ such that $W^r_d(\G)= \emptyset$.
Then $W^r_d(C)=\emptyset$ for  a general
 smooth projective curve $C$ defined  over an algebraically closed field.

In particular, if 
 $d,g,r$ are  such that $\rho^r_d(g)<0$, then there exists such a graph, and hence
 the Brill-Noether Theorem~\ref{BNthm} holds.
\end{prop}
\begin{proof}

As we already observed, the proof of Theorem~\ref{BN+} consists in showing that if $W^r_d(C)$ is non-empty for a general curve $C$, then $W^r_d(\G)$ is non empty for every genus $g$ graph.
Hence the first part of the statement follows.

Now, if $\rho^r_d(g)<0$, the existence of a genus $g$   graph $\G$ for which
$W^r_d(\G)=\emptyset$ is proved in \cite{CDPR}, as stated in Fact~\ref{fact}.
Hence we are done.
\end{proof}

\end{document}